\documentclass[a4paper,fleqn]{cas-sc}

\usepackage[numbers]{natbib}

\def\tsc#1{\csdef{#1}{\textsc{\lowercase{#1}}\xspace}}
\tsc{WGM}
\tsc{QE}
\tsc{EP}
\tsc{PMS}
\tsc{BEC}
\tsc{DE}
\usepackage{amsmath,amsthm}

\newtheorem{theorem}{\bf Theorem}[section]
\newtheorem{remark}{\bf Remark}[section]

\newtheorem{lemma}{\bf Lemma}[section]

\usepackage[utf8]{inputenc} 
\usepackage[T1]{fontenc}    
\usepackage{booktabs}       
\usepackage{amsfonts}       
\usepackage{nicefrac}       
\usepackage{microtype}      
\usepackage{graphicx}
\usepackage{algorithm}
\usepackage{subcaption}
\usepackage{makecell}
\usepackage{amssymb}
\usepackage{color}
\usepackage{appendix}
\usepackage{tikz}

\usepackage[norelsize,boxed,linesnumbered,vlined,algo2e]{algorithm2e}
\usetikzlibrary{positioning}
\usetikzlibrary{arrows}
\usepackage{bm}
\usepackage{multirow}

\newcommand{\red}[1]{\textcolor{black}{#1}}

\hypersetup{
	colorlinks=true, 
	linktoc=all,     
	linkcolor=blue,  
	citecolor=blue,
}

\def \bes{\begin{eqnarray}}
\def \ees{\end{eqnarray}}
\def \bns{\begin{eqnarray*}}
\def \ens{\end{eqnarray*}}

\begin{document}
\let\WriteBookmarks\relax
\def\floatpagepagefraction{1}
\def\textpagefraction{.001}
\shorttitle{PSENet}
\shortauthors{Q. Chen et~al.}

\title [mode = title]{Power series expansion neural network}

\author[1]{Qipin Chen}

\address[1]{Department of Mathematics, Pennsylvania State University, University Park, PA 16802}
\author[1]{Wenrui Hao}\cormark[1]

\author[2]{Juncai He}

\address[2]{Department of Mathematics, The University of Texas at Austin, Austin, TX 78712}

\begin{abstract}
In this paper, we develop a new neural network family based on power series expansion, which is proved to achieve a better approximation accuracy in comparison with existing neural networks. This new set of neural networks embeds the power series expansion (PSE) into the neural network structure. Then it
can improve the representation abilitiy while preserving
comparable computational cost by increasing the degree of PSE instead of
 increasing the
depth or width. Both theoretical approximation and numerical results show the advantages of this new neural network.
\end{abstract}

\begin{highlights}
\item We develop a new set of neural network by introducing  power series expansion;
\item Theoretical analysis shows that PSENet has a better approximation.
\item The new set of neural networks has been tested on different datasets and is shown the advantages;
\end{highlights}


\maketitle

\section{Introduction}

Machine learning has been experiencing an extraordinary resurgence in many important artificial intelligence applications since the late 2000s. In particular, it has produced state-of-the-art accuracy in computer vision, video analysis, natural language processing,  and speech recognition. Recently, interest in machine learning-based approaches in the applied mathematics community has increased rapidly \cite{kang2021ident,wang2021laplacian,zhu2019stop}. This growing enthusiasm for machine learning stems from a massive amount of data available from scientific computations and other sources\red{:} the design of efficient data
analysis algorithms\red{,} advances in high-performance computing and the data-driven modeling \cite{lei2020machine,lu2021deepxde}. To date, there are also many theoretical works on the approximation rate of the neural network, such as 
cosine neural networks \cite{jones1992simple},   sigmoidal neural networks\cite{barron1993universal}, shallow ReLU$^k$ networks  \cite{siegel2020high}, and Deep ReLU networks \cite{lu2020deep,shen2019deep}. 
 However, the main challenge of machine learning is the training process as both complexity and memory \red{requirements} grow rapidly \cite{chen2019homotopy} for deep or wide neural networks. Thus this significant increase in computational cost may not be justified by the performance gain that approximation theories bring.
%
%
\red{Power Series Expansion (PSE)} has been widely used in function approximation and \red{ the resulting linear system can be easily solved even for large-scale computation, for instance, spectral method~\cite{shen2011spectral}.} However, the curse of dimensionality is the main obstacle in the numerical treatment of most high-dimensional problems based on the PSE approximation. In this paper, we will combine the ideas of neural network and PSE to develop a new network, \red{ which we call} PSENet. This new network can achieve a higher accuracy even for shallow or narrow networks.

\section{The formulation of PSENet}
\red{Feed-forward} neural networks, consisting of a series of fully connected layers, can be written as a function from the input $x\in\mathbb{R}^d$ to the output $y\in\mathbb{R}^\kappa$.  Mathematically, \red{such} a neural network with $L$ hidden layers can be written as follows
	\bes	\label{FCNN}
		y(x;\theta) = W_{\red{L+1}}h_{L}+b_{L+1}, ~h_i=\sigma(W_i h_{i-1}+b_i), i\in\{1,\cdots, L\}, \hbox{and }h_0=x,
	\ees
where $W_i\in\mathbb{R}^{\red{d_{i}\times d_{i-1}}}$ is the weight \red{matrix}, $b_i\in\mathbb{R}^{\red{d_{i}}}$ is the bias, $d_i$ is the width of \red{the} $i$-th hidden layer \red{($d_0=d, d_{L+1} = \kappa$)}, and $\sigma$ is the	activation function (for example, the rectified linear unit (ReLU) or the sigmoid activation functions) \red{which is simply applied element-wise on each layer. Moreover,  there is no activation function on the read-out layer since we consider the regression problems.}

\red{Inspired by the power series expansion for a smooth function $f(x)$, namely, $\displaystyle f(x)\approx \sum_{j=0}^n \alpha_j x^j$,
we define each layer  by analogy to a power series expansion.} \red{More precisely, we define a typical PSENet architecture as:
	\bes	\label{PSENet}
	y(x;\theta) = W_{L+1}h_{L}+b_{L+1}, ~h_i=\sum_{j=0}^n{\alpha_{i,j}} \odot \sigma^j (W_i h_{i-1}+b_i), i\in\{1,\cdots, L\}, \hbox{and }h_0=x,
	\ees}
where $\alpha_{i,j} \in \mathbb{R}^{d_i}$ \red{are} the unknown coefficients, \red{$\alpha_{i,j}\odot \sigma^j(W_i h_{i-1}+b_i)$ denoting the Hadamard product of vectors means element-wise multiplication} and $\sigma^j$ stands for $j$-th power of the activation function. Specifically, we define \red{$\sigma^0(x)=x$ as the identity function}. Thus the PSENet is reduced to the ResNet by setting \red{$n=1$} \red{and} \red{$\alpha_{i,0}=\alpha_{i,1}= (1,1,\cdots,1)$}:
	\bes	\label{ResNet}
	h_i=\sigma(W_i h_{i-1}+b_i)+\red{W_ih_{i-1}+b_i}.
	\ees
Thus the PSENet \red{can be thought of as adding} a shortcut \red{(when \red{$\alpha_{i,0} \ne 0$})} or a skip connection that allows information to flow, well just say, more easily from one layer to the next's next layer, i.e., it bypasses data along with normal neural network flow from one layer to the next layer after the immediate next.

\red{For the sake of brevity when discussing the approximation properties of the network in Section~3},
we  define the following general architecture of a hidden layer in PSENet
\begin{equation}\label{PSENet-1}
	h_i=\sum_{j=0}^n \alpha_{i,j}\red{\odot}\sigma^j(W_{i,j} h_{i-1}+b_{i,j}),
\end{equation}
where $W_{i,j}: \mathbb{R}^{d_i \times d_{i-1}}$ \red{and $b_{i,j}\in \mathbb{R}^{d_i}$}. 
On the one hand, it is easy to see that the
original definition in \eqref{PSENet} can be covered by the above formula if
we make weights $W_{i,j}=W_{i}$ and $b_{i,j}=b_{i}$.
\red{On the other hand, one may imagine that the generalized PSENet in~\eqref{PSENet-1} can be reproduced by the PSENet defined in~\eqref{PSENet}
by constructing a block-wise $W_i$ consisting of $(n-1)$ times replicates of $W_{i,j}$ and an appropriate $\alpha_{i,j}$ with a special sparse structure.}
As a consequence, we have the following theorem to show the equivalence between \red{the} two formulas.

\begin{theorem}\label{thm:equivalence}
\red{Let $f(x): \mathbb{R}^d \mapsto \mathbb{R}^\kappa$ be a \red{generalized} PSENet model defined}
	by \eqref{PSENet-1} with hyper-parameters maximal power $n$ and \red{widths} $d_i$ \red{for i=1:L}. There exists
	a PSENet model $\tilde f(x)$ defined by \eqref{PSENet} with hyper-parameters maximal power $n$
	 and \red{widths} $\tilde d_i = \red{(n+1)} d_i$ \red{for i=1:L,} \red{such that $\tilde f(x) = f(x)$}.
\end{theorem}

\begin{proof}
	\red{As defined in \eqref{PSENet-1}}, the \red{generalized} PSENet function $f(x)$ \red{has the form of}
$f(x) = W_{L+1} h_L(x) + b_L$ where
	\begin{equation}\label{key}
		h_i(x)=\sum_{j=0}^n\alpha_{i,j}\red{\odot}\sigma^j(W_{i,j} h_{i-1}(x)+b_{i,j}), \quad i = 1,\cdots, L,
	\end{equation}
	where $h_0(x) = x$ and \red{$W_{L+1} \in \mathbb{R}^{\kappa \times d_L}$}.
	For simplicity, we denote \red{$A_{i,j} = {\rm Diag}(\alpha_{i,j})\in \mathbb{R}^{d_i \times d_i}$ as the diagonal matrix obtained by taking $\alpha_{i,j}$ as the diagonal elements. It follows that $\alpha_{i,j}\red{\odot}\sigma^j(W_{i,j} h_{i-1}(x)+b_{i,j}) = A_{i,j}\sigma^j(W_{i,j} h_{i-1}(x)+b_{i,j})$}
\red{According to \eqref{PSENet}}, we denote $\tilde f(x) = \tilde W_{L+1} \tilde h_L(x) + \red{\tilde b_{L+1}}$ where
	\begin{equation}\label{key}
	\tilde h_i(x)=\sum_{j=0}^n\tilde \alpha_{i,j}\red{\odot}\sigma^j(\tilde W_{i} \tilde h_{i-1}(x)+\tilde b_{i}), \quad i = 1,\cdots,L,
\end{equation}
$\tilde h_0(x) = x$ and \red{$\tilde W_{L+1} \in \mathbb{R}^{\kappa \times (n+1)d_L}$}.
Now, we construct $\tilde f(x)$ by taking
\red{
\begin{equation}\label{key}
\tilde W_{i} = \begin{pmatrix}
W_{i,0}A_{i-1,0} & W_{i,0}A_{i-1,1}& \cdots&W_{i,1}A_{i-1,n} \\
W_{i,1}A_{i-1,0} & W_{i,1}A_{i-1,1} & \cdots &W_{i,1}A_{i-1,n} \\
\vdots & \vdots & \vdots & \vdots \\
W_{i,n}A_{i-1,0} & W_{i,n}A_{i-1,1} & \cdots &W_{i,n}A_{i-1,n}
\end{pmatrix} \in \mathbb{R}^{(n+1)d_i \times (n+1)d_{i-1}},
\tilde b_{i} = \begin{pmatrix}
	b_{i,0} \\
	b_{i,1}  \\
	\vdots  \\
	b_{i,n}
\end{pmatrix}\in \mathbb{R}^{(n+1)d_i},
\end{equation}
\begin{equation}
\text{ and }
    \tilde \alpha_{i,j} = \big(\underbrace{0,\cdots,0}_{(j-1)d_i}, \underbrace{1,\cdots,1}_{d_i} ,\underbrace{0, \cdots ,0}_{(n+1-j)d_i}
\big)^T \in \mathbb{R}^{(n+1)d_i},
\end{equation}}
for $i = 2,\cdots, L$. 
In addition, we take
\begin{equation}\label{key}
		\tilde W_{1} = \begin{pmatrix}
		W_{1,0}\\
		W_{1,1}  \\
		\vdots  \\
		W_{1,n}
	\end{pmatrix},
\tilde b_{1} = \begin{pmatrix}
	b_{1,0} \\
	b_{1,1}  \\
	\vdots  \\
	b_{1,n}
\end{pmatrix},
\tilde W_{L+1} = \begin{pmatrix}
		W_{L+1}A_{L,0} \\ W_{L+1}A_{L,1} \\ \vdots \\ W_{L+1}A_{L,n}
	\end{pmatrix}^T,
	\text{ and }
	\tilde b_{L+1} = b_{L+1}.
\end{equation}
Then, we can finish the proof by \red{showing} that
\begin{equation}\label{key}
	\tilde h_i(x) =
	\begin{pmatrix}
		[\tilde h_i(x)]_{0} \\
		[\tilde h_i(x)]_{1}  \\
		\vdots  \\
		[\tilde h_i(x)]_{n}
	\end{pmatrix} =
	\begin{pmatrix}
	\sigma^1(W_{i,0}h_{i-1}(x) + b_{i,0}) \\
		\sigma^2(W_{i,1}h_{i-1}(x) + b_{i,1})  \\
	\vdots  \\
		\sigma^n(W_{i,n}h_{i-1}(x) + b_{i,n})
	\end{pmatrix},
\end{equation}
for $i = 1,\cdots, L$.
In fact, for $i=1$, we have
\begin{equation}\label{key}
	\tilde h_1(x) =
\sum_{j=0}^n\tilde \alpha_{1,j}\red{\odot}\sigma^j (\tilde W_1 x + \tilde b_1)=
	\begin{pmatrix}
		\sigma^1(W_{1,0}x + b_{1,0}) \\
		\sigma^2(W_{1,1}x + b_{1,1}) \\
		\vdots  \\
		\sigma^n(W_{1,n}x + b_{1,n})
	\end{pmatrix}.
\end{equation}
Then, by induction we have
\begin{equation}\label{key}
	\begin{aligned}
	&\tilde h_i(x) =
	\sum_{j=0}^n\tilde \alpha_{i,j} \red{\odot}\sigma^j (\tilde W_i \tilde h_{i-1}(x)  + \tilde b_i)=
	\begin{pmatrix}
		\sigma^1(W_{i,0}\sum_{j=0}^n A_{i-1,j}[\tilde h_{i-1}(x)]_j + b_{i,0}) \\
		\sigma^2(W_{i,1}\sum_{j=0}^n A_{i-1,j}[\tilde h_{i-1}(x)]_j + b_{i,1}) \\
		\vdots  \\
		\sigma^n(W_{i,n}\sum_{j=0}^n A_{i-1,j}[\tilde h_{i-1}(x)]_j + b_{i,n})
	\end{pmatrix} \\
&=\begin{pmatrix}
	\sigma^1\left(W_{i,0}\left(\sum_{j=0}^n A_{i-1,j} \sigma^j(W_{i-1,j}h_{i-2}(x) + b_{i-1,j})\right)+ b_{i,0}\right) \\
	\sigma^2\left(W_{i,1}\left(\sum_{j=0}^n A_{i-1,j} \sigma^j(W_{i-1,j}h_{i-2}(x) + b_{i-1,j})\right) + b_{i,1}\right) \\
	\vdots  \\
	\sigma^n\left(W_{i,n}\left(\sum_{j=0}^n A_{i-1,j} \sigma^j(W_{i-1,j}h_{i-2}(x) + b_{i-1,j})\right) + b_{i,n}\right)
\end{pmatrix} = 	\begin{pmatrix}
\sigma^1(W_{i,0}h_{i-1}(x) + b_{i,0}) \\
\sigma^2(W_{i,1}h_{i-1}(x) + b_{i,1})  \\
\vdots  \\
\sigma^n(W_{i,n}h_{i-1}(x) + b_{i,n})
\end{pmatrix}
	\end{aligned}.
\end{equation}
Therefore, we have
\begin{equation}\label{key}
	\tilde f(x) = \tilde W_{L+1}\tilde h_L(x)
	+\red{\tilde b_{L+1}}= W_{L+1} \sum_{j=0}^n \alpha_{L,j} \red{\odot}\sigma^j(W_{L,j}h_{L-1}(x) + b_{L,j}) + \red{b_{L+1}} = W_{L+1} h_{L}(x) + \red{b_{L+1}}= f(x),
\end{equation}
which \red{finishes} the proof.
\end{proof}

\section{\red{Representation abilities} and approximation properties of PSENet}
In this section, we will discuss the \red{representation abilities} and approximation power of PSENet defined in (\ref{PSENet-1}) in comparison
with classical DNN under the ReLU activation function, ${\rm ReLU}(x)$, i.e.
\begin{equation}
	\sigma(x) = {\rm ReLU}(x) := \max \{ 0,x\}, \quad x \in \mathbb{R}.\nonumber
\end{equation}

\subsection{\red{One-hidden-layer PSENet}}
\red{Given Theorem~\ref{thm:equivalence}, we consider a typical generalized PSENet function (i.e. $\kappa = 1$) as
$$
f(x) = W_2 \left(  \sum_{j=0}^n\alpha_{1,j}\odot\sigma^j(W_{1,j}x+b_{1,j})\right).
$$
Since $W_2 \in \mathbb{R}^{1\times d_i}$, we can merge $W_1$ and $\alpha_{1,j}$ and rewrite $f(x)$ as
$$
f(x) = \sum_{j=0}^n\bar \alpha_{1,j} \sigma^j(W_{1,j}x+b_{1,j}),
$$
where $\bar \alpha_{1,j}\in \mathbb R^{1\times d_1}$ and $\bar \alpha_{1,j} \sigma^j(W_{1,j}x+b_{1,j})$ is defined by standard matrix multiplication.
Then, we may further generalize the above equation by taking $W_{1,j}\in \mathbb{R}^{m_j \times d}$ for $j=0,\cdots, n$ and for any $(m_0, m_1, \cdots, m_n) \in \mathbb{N}^n$. 
Then we denote the set of generalized one-hidden-layer PSENet function as
\begin{equation}\label{Vmn}
	V_{\bm m}^n = \left\{  f(x) = \sum_{j=0}^n \alpha_j \sigma^{j}(W_j x + b_j)~:~ W_j \in \mathbb{R}^{m_j \times d}, b_j \in \mathbb{R}^{m_j}, \alpha_j \in \mathbb{R}^{1\times m_j}\right\},
\end{equation}
for any $\bm m = (m_0, m_1, \cdots, m_n) \in \mathbb{N}^n$. Here we notice that there exist $\omega \in \mathbb R^d$ and $b\in \mathbb R$ such that $\alpha_0 \sigma^{0}(W_0 x + b_0) = \omega \cdot x + b$ is only a linear function on $\mathbb R^d$ no matter how big $m_0$ is. Thus, we always assume $m_0=1$ in $V_{\bm m}^n$.}

Next we show the \red{representation abilities} and approximation power in terms of the largest power $n$ and the number of neurons $\displaystyle|\bm m| = \sum_{j=0}^n m_j$.
\red{Since the activation function and its powers are ${\rm ReLU}^k$, it is natural to consider the connections between the PSENet $V_{\bm m}^n$ and the 
B-Spline function space on one-dimensional space.}
According to \cite{de1971subroutine},
a \red{one-dimensional} cardinal B-Spline of degree $n\ge 0$
denoted by $b^n(x)$ for \red{$x\in \mathbb R$}, can be written as
\begin{equation}\label{splinetorelu}
	b^n(x)=(n+1)\sum_{i=0}^{n+1} w_i\sigma^n(i-x) \text{~and~} w_i={\displaystyle\prod_{j=0,j\neq i}^{n+1}} \frac{1}{i-j},
\end{equation}
where $b^{n}(x)$ is supported on $x\in[0,n+1] \subset \mathbb R$ and $n\geq 1$. Moreover,  the cardinal B-Spline series of degree $n$ on the uniform grid with mesh size $h=\frac{1}{k+1}$ is defined as
\begin{equation}\label{Bkn}
	B_k^n=\Big\{v(x)=\sum_{j=-n}^{k}	c_jb^n_{j,h}(x)\Big\} \hbox{~where~} 	b^n_{j,h}(x)=b^{n}(\frac{x}{h}-j).
\end{equation}
Then we have the following lemma for the \red{representation abilities of $V_{\bm m}^n$ in terms of its connections with B-Spline function space. }
\begin{lemma}\label{lem:B-V} By choosing  \red{$m_i \ge k_i+i+1$ ($i=1,\cdots,n$)}, we have
	\begin{equation}
		\label{SV}
		\displaystyle\bigcup_{i=1}^n B_{k_i}^i\subset V_{\bm m}^n,
	\end{equation}
where $V_{\bm m}^n$ and $B_{k_i}^i$ are defined by \eqref{Vmn} and
	\eqref{Bkn}, respectively.
\end{lemma}
\begin{proof}

We consider the so-called finite neuron methods~\cite{xu2020finite} with ${\rm ReLU}^k$ as the activation function and define
the one hidden layer neural network described in~\cite{xu2020finite} as
\begin{equation}\label{key}
	{V}_m^n := \left\{ f(x) ~:~ f(x) = \sum_{j=1}^m a_j \sigma^n(\omega_j \cdot x + b_j) \right\}.
\end{equation}
Obviously, we have
\begin{equation}\label{eq:VmVm}
V_{\bm m}^n  = \bigcup_{i=0}^n V_{m_i}^i.
\end{equation}
Lemma 3.2 in \cite{xu2020finite} shows that
\begin{equation}\label{eq:BmVm}
	B_{m}^i \subset {V}_{m+i+1}^i,
\end{equation}
then we complete the proof by combining \eqref{eq:VmVm} and \eqref{eq:BmVm}.
\end{proof}

\red{As a result of the above lemma}, we
have the following approximation result for $V_{\bm m}^n$.
\begin{theorem}[1D case]
Suppose $u\in H^{n+1}(\Omega)$ for a bounded domain $\Omega\subset \mathbb R$, we have
\begin{equation}
	\label{SVerror}
	\inf_{v\in V_{\bm m}^n}\|u-v\|_{s,\Omega}
	\lesssim \min_{i=1,2,\cdots, n} \left\{ m_i^{s-(i+1)} \|u\|_{i+1,\Omega} \right\},
\end{equation}
for any large enough $m_i > i+1$. 
\red{Here $H^{k}(\Omega)$ (or $W^{k,2}(\Omega)$) denotes the standard Sobolev space~\cite{adams2003sobolev} on $\Omega$ with norm $\|u\|_{k,\Omega}$.}
\end{theorem}

\begin{proof}
  According to the error estimate of $B^i_{N}$ in~\cite{xu2020finite}, we have
  \begin{equation}
  	\inf_{v\in  B_{m_i-i-1}^i}\|u-v\|_{s,\Omega}  \lesssim m_i^{s-(i+1)}\|u\|_{i+1,\Omega}.
  \end{equation}
In addition, we have
\begin{equation}
		\displaystyle\bigcup_{i=1}^n B_{m_i - i-1}^i\subset V_{\bm m}^n,
\end{equation}
if $m_i >i+1$ in Lemma~\ref{lem:B-V}.
This indicates that
\begin{equation}
	\inf_{v\in V_{\bm m}^n}\|u-v\|_{s,\Omega} \le \inf_{v\in \cup_{i=1}^n B_{m_i-i-1}^i}\|u-v\|_{s,\Omega} \lesssim \min_{i=1,2,\cdots, n} \left\{ m_i^{s-(i+1)} \|u\|_{i+1,\Omega} \right\}.
\end{equation}

\end{proof}
\begin{remark}
\red{When} comparing with ${\rm ReLU}^n$-DNN ~\cite{xu2020finite}, the PSENet has the following
advantages:
\begin{enumerate}
	\item If we have no information about the regularity of the target function $u(x)$ a priori, the PSENet $V_{\bm m}^n$ gives
	an adaptive and uniform scheme for approximating any $u \in H^i(\Omega)$ for all $i \ge 1$. However, ${\rm ReLU}^n$-DNN can only
	work for $u \in H^i(\Omega)$ for $i \ge n$.
	\item 	By choosing $m_i = 0$ for $i < n$,  the PSENet $V_{\bm m}^n$ \red{recovers} the ${\rm ReLU}^n$-DNN exactly. Thus if $u(x) \in H^n(\Omega)$, then PSENet provides almost the same asymptotic convergence rate in terms of
	the number of hidden neurons $|\bm m|$ as the ${\rm ReLU}^n$-DNN ~\cite{xu2020finite}.

	\item If $u(x)$ is a smooth function,
	the PSENet $V_{\bm m}^n$ then provides a better approximation  than
 ${\rm ReLU}^n$-DNN when the number of neurons, $m$, is not large since
	$\|u\|_{i+1,\Omega}$ might be very large but $\left(\frac{|\bm m|}{n}\right)^{s-(i+1)}$ is not small enough.
\end{enumerate}
\end{remark}

Following the observation of \red{Lemma (3.10)} in \cite{xu2020finite}, we have
the next theorem about the \red{representation 
abilities of the PSENet in terms of its connections with polynomials on the multi-dimensional space}.
\begin{theorem}[Multi-dimensional case]\label{miltidim_spectral}For any polynomial $p(\bm x) = \displaystyle\sum_{|\alpha| \le k} a_\alpha \bm x^\alpha$ on $\mathbb{R}^d$,
	there exists a PSENet function $\hat p(x) =\displaystyle \sum_{j=0}^{k} c_j \sigma^{j}(W_j x + b_j)$ with $m_i \le 2 \binom{i+d-1}{i}$,
	such that
	\begin{equation}\label{key}
		\hat p(x) = p(x),
	\end{equation}
on $\mathbb{R}^d$.
\end{theorem}
\begin{proof}
\red{Given} the connections between PSENet and $ReLU^k$-DNN in \eqref{eq:VmVm},
we only need to prove 
\red{$$
\left\{ \sum_{|\alpha| =i } a_\alpha \bm x^\alpha ~:~ a_\alpha \in \mathbb R\right\} \subset V_{m_i}^i\hbox{~and~} m_i= 2 \binom{i+d-1}{i}.
$$}
To prove that, we first recall the following property that
\begin{equation}\label{key}
	x^i = {\rm ReLU}^i(x) + (-1)^{i} {\rm ReLU}^i(-x).
\end{equation}
Moreover $d_i = \binom{i+d-1}{i}$ is the dimension of the space of
homogeneous polynomials on $\mathbb{R}^d$ with degree $i$.
Thus, we only need to prove that we can choose suitable $w_s \in \mathbb{R}^d$ for $s = 1:d_i$ such that
\begin{equation}\label{key}
(w_s \cdot \bm x)^i  = {\rm ReLU}^i(w_s \cdot x) + (-1)^{i} {\rm ReLU}^i(-w_s \cdot x) \in \text{PSENet},
\end{equation}
 forms a basis for homogeneous polynomials on $\mathbb{R}^d$ with degree $i$.
By denoting
\begin{equation}\label{key}
\bm X = (\bm x^{\alpha_1}, \bm x^{\alpha_2}, \cdots, \bm x^{\alpha_{d_i}})^T,
\end{equation}
as the natural basis for the space of homogeneous polynomials on $\mathbb{R}^d$ with degree $i$,
we have
\begin{equation}\label{key}
	((w_1 \cdot \bm x)^i, (w_2 \cdot \bm x)^i, \cdots, (w_{d_i} \cdot \bm x)^i   )^T = \bm W \bm X,
\end{equation}
where $\bm W \in \mathbb{R}^{d_i \times d_i}$ is a matrix formed by $w_1, w_2,\cdots, w_{d_i}$.
Based on the generalized Vandermonde \red{determinant} identity~\cite{yaacov2014multivariate}, we see that
$\bm W$ is an invertible matrix if we choose \red{appropriate} $w_s$.
Therefore $((w_1 \cdot \bm x)^i, (w_2 \cdot \bm x)^i, \cdots, (w_{d_i} \cdot \bm x)^i   )$ forms the basis for the
space of homogeneous polynomials on $\mathbb{R}^d$ with degree $i$.
\red{A more comprehensive description about how to choose $w_s$ to get an invertible $\bm W$ will be available in \cite{he2021dnn}.}
\end{proof}
\begin{remark}
\begin{enumerate}
  \item The total number of neurons is
\begin{equation}
	|\bm m| = \sum_{i=0}^k m_i =  \sum_{i=0}^k 2 \binom{i+d-1}{i} = 2 \binom{k+d}{k},\nonumber
\end{equation}
which equals the number of neurons of ${\rm ReLU}^k$-DNN to recover
polynomials with degree $k$ as shown in \cite{xu2020finite}.
Considering the spectral accuracy of polynomials for smooth functions in terms of
the degree $k$, the above representation theorem shows that PSENet can achieve
an exponential approximation rate for smooth functions with respect to $n$ in $V_{\bm m}^n$,
some similar results can be found in~\cite{e2018exponential, he2021approximation,opschoor2020deep,opschoor2019exponential,
	tang2019chebnet}

  \item PSENet can take a large degree $n$ to reproduce
high order polynomials instead of a deep network. But other networks need deep layers to improve the performance,
for example expressive power~\cite{arora2018understanding, he2021relu,he2020relu,montufar2014number, shen2019nonlinear,telgarsky2016benefits},
approximation properties~\cite{e2018exponential,guhring2020error,lu2017expressive, montanelli2019deep, opschoor2020deep,opschoor2019exponential,poggio2017and},
benefits for training~\cite{arora2018optimization} and etc.
\item \red{The main results in this subsection are established by combining the representation abilities of the ${\rm ReLU}^k$-DNN  $V^{k}_{m}$, i.e., its connection with B-Spline (Lemma 3.2 in \cite{xu2020finite}) and polynomials (Lemma 3.10 in \cite{xu2020finite}). Moreover, we also reveal the natural relation between the PSENet $V^{n}_{\bm m}$ and the ${\rm ReLU}^k$-DNN $V^{k}_{m}$  in \eqref{eq:VmVm}. In the following subsection, we show another core feature of the PSENet $V^{n}_{\bm m}$ that it can approximate singular function with optimal rate. This has not been studied in~\cite{xu2020finite}, and the PSENet can achieve a better approximation rate compared to the results presented in both ~\cite{he2021approximation} and~\cite{opschoor2020deep}.}
\end{enumerate}
\end{remark}

\subsection{Optimal approximation rate on singular functions by using PSENet}
We apply the PSENet  on the singular function approximation which has been widely studied in hp-FEM \cite{babuvska1994p, schwab1998numerical} and
 consider non-smooth functions in Gevrey class~\cite{chernov2011exponential, schwab1998numerical, opschoor2020deep} on $I = (0,1)$:
For any $\beta>0$, we define \red{the} function $\varphi_\beta (x) = x^{\beta}$ on $[0,1]$, the seminorm as
\begin{equation}
	|u|_{H_{\beta}^{k,\ell}(I)}:=||\varphi_{\beta+k-\ell}D^ku||_{L^2(I)},
\end{equation}
and the $H^{k,\ell}_\beta$ norm as
\begin{equation}
	||u||^2_{H_{\beta}^{k,\ell}(I)}:=\left\{
	\begin{split}
		&\sum\limits_{k'=0}^k |u|^2_{H_{\beta}^{k',0}(I)}, &if ~\ell=0,\\
		&\sum\limits_{k'=l}^k |u|^2_{H_{\beta}^{k',\ell}(I)}+||u||^2_{H^{\ell-1}(I)}, & if~ \ell\ge 1,
	\end{split}
	\right.
\end{equation}
where $\ell, k = 0,1,2,\cdots$.
For any $\delta\ge 1$ the Gevrey class $\mathcal{G}_{\beta}^{\ell,\delta}(I)$ is defined as the class of
functions $u\in \cap_{k\ge l}H_\beta^{k,l}(I)$ for which there exist $M,m>0$, such that
\begin{equation}
	\forall k\ge l:|u|_{H_{\beta}^{k,l}(I)}\le Mm^{k-l}((k-l)!)^{\delta}.
\end{equation}

When $d=1$, these function classes have a singular point at $x=0$, then \red{the} $hp$ finite element method has exponential convergence to \red{this} function class \cite{schwab1998numerical, gui1986theh}.

We consider the piece-wise polynomial space on mesh $\mathcal{T}_n: 0 = x_0<\cdots<x_n=1$ as
\begin{equation}
	\begin{split}
		P_{\bm p}(\mathcal{T}_n) = \{ p_h \mbox{ is continuous on } I ~|~p_h\mbox{ is a polynomial on grid }[x_{i-1},x_i]
		\mbox{ with degree } p^{(i)}  \},
	\end{split}
\end{equation}
and have the following estimate:
\begin{lemma}[\cite{opschoor2020deep}]\label{HPresult}
	Let $\sigma,\beta\in(0,1),\delta\ge 1,u\in \mathcal{G}_{\beta}^{2,\delta}(I)$ and $N\in \mathbb{N}$ be given. For $\mu_0=\mu_0(\sigma,\delta,m):=\max\{1,\frac{m}{2}(2e)^{1-\delta}\}$ and for any $\mu>\mu_0$, let $\bm p=(p^{(i)})_{i=1}^n\subset\mathbb{N}$ be defined as $p^{(1)}:=1$ and $p^{(i)}:=\lfloor\mu i^{\delta}\rfloor$ for $i\in\{2,...,n\}$.
	Then there exists $v(x) \in P_{\bm p}(\mathcal{T}_n)$ with $v(x_i) = u(x_i)$ and $x_i = \frac{1}{2^{n-i}}$ for $i\in \{1,...,n\}$
	such that for constants $C(\sigma, \beta, \delta, \mu, M, m), c(\beta, \delta)>0$ it holds that
	\begin{equation}
		||u-v||_{H^1(0,1)} \le C e^{-cn}.
	\end{equation}
\end{lemma}

Then we have the following lemma about the decomposition
properties of functions of PSENet in $P_{\bm p}(\mathcal{T}_n)$ with $p^{(i)} \le p^{(i+1)}$.
\begin{lemma}\label{lem:ph}
For any function $p_h \in P_{\bm p}(\mathcal{T}_n)$ with $p^{(i)} \le p^{(i+1)}$,
$p_h(x)$ can be reproduced by a \red{one-hidden-layer} PSENet, namely,
\begin{equation}\label{key}
	p_h(x) = \sum_{j=0}^{p^{(n)}} \alpha_j \sigma^{j}(W_j x + b_j), \quad \forall x \in [0,1]
\end{equation}
where $m_j \le n$.
\end{lemma}
\begin{proof}
First, we can write $p_h(x)$ as
\begin{equation}\label{key}
	p_h(x) - p_h(0) = \sum_{i=1}^n \chi_{I_i}(x) p_{h,i}(x),
\end{equation}
where $\chi_{I_i}(x)$ is the indicator function of $I_i = [x_{i-1}, x_i)$ for $i=1,\cdots, n-1$, and $I_n = [x_{n-1}, x_n]$. Here $p_{h,i}(x)$ is the polynomial of $p_h(x)$ on $I_i$ with degree $p^{(i)}$. Thanks to the property that $p^{(i)} \le p^{(i+1)}$, we re-write $p_h(x)$ as
\begin{equation}\label{key}
	p_h(x) - p_h(0) = \sum_{i=1}^n \chi_{\tilde I_i}(x) \tilde p_{h,i}(x),
\end{equation}
where $\tilde I_i = [x_{i-1}, 1]$ and $\tilde p_{h,i}(x)$ is a polynomial of degree $p^{(i)}$ defined as
\begin{equation}\label{key}
	\tilde p_{h,i}(x) = p_{h,i}(x) - \tilde p_{h,i-1}(x), \quad i = 2, 3, \cdots,n,
\end{equation}
with $\tilde p_{h,1}(x) = p_{h,1}(x)$. In addition, we have
\begin{equation}\label{key}
	\tilde p_{h,i}(x) = \sum_{j=1}^{p^{(i)}} \tilde a^{(i)}_j(x-x_{i-1})^j,
\end{equation}
\red{because of} the continuity of $p_{h}(x)$ on $[0,1]$.
Based on the above property and definition of indicator function $\chi_{\tilde I_i}(x)$, we have
\begin{equation}\label{key}
\chi_{\tilde I_i}(x) \tilde p_{h,i}(x) = \sum_{j=1}^{p^{(i)}} \tilde a^{(i)}_j \sigma^j (x-x_{i-1}),
\end{equation}
on $[0,1]$.
That finishes the proof.
\end{proof}
This lemma \red{is different from} results in~\cite{opschoor2020deep, xu2020finite, he2021approximation} since it
can achieve an exact representation formula for any piecewise polynomials while \red{their results} can only construct
the representation globally or establish some approximation results for \red{the} piecewise case.

Based on these two lemmas, we have the following main theorem for approximation property of
PSENet for Gevrey class.
\begin{theorem}\label{ConvergeResult}
	For all $\delta\ge 1$, $\beta\in(0,1)$, and $u\in \mathcal{G}_{\beta}^{2,\delta}(I)$, there exists
	a PSENet function $\hat u(x)$ with one hidden layer
	 such that
	\begin{equation}
		||u-\hat u||_{H^1(0,1)}\le C_0e^{-C_1 |\bm m|^{\frac{1}{\delta+1}}},
	\end{equation}
	where \red{$ \displaystyle|\bm m|=\sum_{j=0}^{p^{(n)}} m_j$} for $\hat u(x)$, and $C_0(\sigma, \beta, \delta, \mu, M, m)$ and $C_1(\beta, \delta)$ only depend on the function $u(x)$ \red{,} similar to Lemma~\ref{HPresult}.
\end{theorem}
\begin{proof}

 For any function
$u(x) \in  \mathcal{G}_{\beta}^{2,\delta}(I)$, Lemma~\ref{HPresult} shows that
there exists $u_h(x) \in P_{\bm p}(\mathcal{T}_n)$
such that
\begin{equation}\label{key}
	||u-u_h||_{H^1(0,1)} \le C e^{-cn},
\end{equation}
with $p^{(1)}:=1$ and $p^{(i)}:=\lfloor\mu i^{\delta}\rfloor$ for $i\in\{2,...,n\}$.
According to Lemma~\ref{lem:ph}, there exists a PSENet function $\hat u(x)$ with
one hidden layer and $m_j \le p^{(n)}-j+1$ such that $\hat u(x) = u_h(x)$ on $[0,1]$,
i.e.
\red{
\begin{equation}\label{key}
\|u - \hat u\|_{H^1(0,1)}	= \|u-u_h\|_{H^1(0,1)}  \le C e^{-cn}.
\end{equation}
}
Then, it is easy to obtain the final approximation rate since
\begin{equation}\label{key}
	|\bm m| = \sum_{j=0}^{p^{(n)}} m_j \lesssim \mu n^{\delta + 1}.
\end{equation}

\end{proof}

This approximation result \red{achieves} a better convergence rate in comparison with
results in \cite{he2021approximation} and ~\cite{opschoor2020deep}, whose
rates are $C_0e^{-C_1M^{\frac{1}{2\delta+1}}}$ or $C_0e^{-C_1M^{\frac{1}{3\delta+1}}}$, respectively.
Furthermore, this result is optimal since it shares the same order with the most general approximation result of piecewise polynomials.

\section{Numerical results}
In this section,  we compare the PSENet with ResNet on both fully connected and convolutional neural networks by using the ReLU activation function.

\subsection{Function approximation}
We first compare the PSENet with fully connected neural networks and ResNet to approximate $y=sin(n\pi x)$ on $[0,1]$ and $y=\sin(n\pi (x_1+x_2))$ on $[0,1]\times[0,1]$ with \red{single, two and three hidden layers.} 
 We train the neural networks with uniform grid points with a $0.01$ mesh size for both functions and use Adam training algorithm with a fixed learning rate $0.01$.
The training loss for different neural networks are compared and \red{shown} in Table \ref{table:1d_1hidden} which demonstrates that PSENet has a better approximation ability in comparison with the other two networks with the optimal degree shown \red{.}  
%
Secondly, we consider \red{a} 1D function $f(x)=x^\alpha$ with $\alpha\in(0,1)$ on $x\in[0,1]$, where $x=0$ is a singularity. \red{From} the theoretical analysis in Section 3.2, the PSENet can achieve \red{a better} approximation rate \red{when compared to} ReLU$^k$-DNN which is confirmed in Table \ref{table:singular}.

\begin{table}[h]
		\scriptsize
		\centering
	\caption{The comparison between PSENet and fully connected neural networks and ResNet on the training loss of $f(x)=\sin(n\pi x)$ on [0,1] and $y=\sin(n\pi (x_1+x_2))$ on $[0,1]\times[0,1]$. The number of neurons on each layer is 10. The best approximation accuracy for PSENet with different degree $n$ is highlighted. }
		\begin{tabular}{|c|c|c|c|c|c|c|c|c|}
			\toprule
			& Function & FC &  ResNet & \multicolumn{5}{|c|}{PSENet}\\
			\cmidrule{5-9}
			& &  & & n=1 & n=2 & n=3 & n=4 & n=5\\
			\midrule
			\multirow{3}{*}{1-hidden-layer} &$\sin(3\pi x)$ & $2\times 10^{-1}$ & $2\times 10^{-1}$ & $2\times 10^{-1}$ & $1\times 10^{-1}$ & $2\times 10^{-1}$ & $1 \times 10^{-1}$ & $\bm{6 \times 10^{-3}}$\\
			  \cmidrule{2-9}
			   &$\sin(4\pi x)$ & $3\times 10^{-1}$ & $2\times 10^{-1}$ & $3\times 10^{-1}$ & $4\times 10^{-1}$ & $2\times 10^{-1}$ & $\bm{1 \times 10^{-1}}$ & $2 \times 10^{-1}$\\
			   \cmidrule{2-9}
			   &$\sin(5\pi x)$ & $2\times 10^{-1}$ & $3\times 10^{-1}$ & $2\times 10^{-1}$ & $2\times 10^{-1}$ & $3\times 10^{-1}$ & $1 \times 10^{-1}$ & $\bm{5 \times 10^{-2}}$\\
			\midrule
			\multirow{3}{*}{2-hidden-layer} &  $\sin(3\pi x)$ & $3\times 10^{-3}$ & $4\times 10^{-3}$ & $3\times 10^{-3}$ & $2\times 10^{-1}$ & $\bm{2\times 10^{-3}}$ & $4 \times 10^{-3}$ & $3 \times 10^{-3}$\\
			  \cmidrule{2-9}
			   &$\sin(4\pi x)$ & $2\times 10^{-1}$ & $3\times 10^{-1}$ & $2\times 10^{-1}$ & $3\times 10^{-1}$ & $1\times 10^{-2}$ & $\bm{3 \times 10^{-4}}$ & $2 \times 10^{-2}$\\
\cmidrule{2-9}
			   &$\sin(5\pi x)$ & $2\times 10^{-1}$ & $1\times 10^{-1}$ & $1\times 10^{-1}$ & $1\times 10^{-1}$ & $1\times 10^{-1}$ & $3 \times 10^{-1}$ & $\bm{4 \times 10^{-3}}$\\
			   \midrule
			\multirow{3}{*}{3-hidden-layer} &  $\sin(3\pi x)$ & $2\times 10^{-1}$ & $1\times 10^{-3}$ & $3\times 10^{-3}$ & $1\times 10^{-3}$ & $3\times 10^{-3}$ & $8 \times 10^{-4}$ & $\bm{6 \times 10^{-4}}$\\
			  \cmidrule{2-9}
			   &$\sin(4\pi x)$ & $3\times 10^{-1}$ & $3\times 10^{-1}$ & $2\times 10^{-1}$ & $3\times 10^{-1}$ & $\bm{1\times 10^{-2}}$ & $2 \times 10^{-1}$ & $1 \times 10^{-1}$\\
\cmidrule{2-9}
			   &$\sin(5\pi x)$ & $3\times 10^{-1}$ & $1\times 10^{-1}$ & $3\times 10^{-1}$ & $1\times 10^{-1}$ & $\bm{3\times 10^{-3}}$ & $1 \times 10^{-1}$ & $1 \times 10^{-1}$\\
			 \multirow{3}{*}{2-hidden-layer}  & $\sin(3\pi (x_1+x_2))$ & $2\times 10^{-1}$ & $7\times 10^{-2}$ & $2\times 10^{-2}$ & $1\times 10^{-2}$ & \bm{$6\times 10^{-3}}$ & $7 \times 10^{-2}$ & $5 \times 10^{-2}$\\
			  \cmidrule{2-9}
			   &$\sin(4\pi (x_1+x_2))$ & $4\times 10^{-1}$ & $1\times 10^{-1}$ & $2\times 10^{-1}$ & $5\times 10^{-2}$ & $4\times 10^{-1}$ & $\bm{1 \times 10^{-2}}$ & $4 \times 10^{-1}$\\
			    \cmidrule{2-9}
			   &$\sin(5\pi (x_1+x_2))$ & $3\times 10^{-1}$ & $4\times 10^{-1}$ & $4\times 10^{-1}$ & $4\times 10^{-1}$ & $4\times 10^{-1}$ & $4 \times 10^{-1}$ & \bm{$2 \times 10^{-1}$}\\
			   \midrule
			   \multirow{3}{*}{3-hidden-layer}  & $\sin(3\pi (x_1+x_2))$ & $1\times 10^{-1}$ & $1\times 10^{-1}$ & $1\times 10^{-1}$ & $\bm{2\times 10^{-3}}$ & $2\times 10^{-2}$ & $8 \times 10^{-3}$ & $1 \times 10^{-2}$\\
			  \cmidrule{2-9}
			   &$\sin(4\pi (x_1+x_2))$ & $3\times 10^{-1}$ & $1\times 10^{-1}$ & $1\times 10^{-1}$ & $8\times 10^{-2}$ & $\bm{7\times 10^{-2}}$ & $1 \times 10^{-1}$ & $1 \times 10^{-1}$\\
			    \cmidrule{2-9}
			   &$\sin(5\pi (x_1+x_2))$ & $3\times 10^{-1}$ & $4\times 10^{-1}$ & $4\times 10^{-1}$ & $3\times 10^{-1}$ & $3\times 10^{-1}$ & $3 \times 10^{-1}$ & \bm{$5 \times 10^{-2}$}\\
			\bottomrule
			\bottomrule
		\end{tabular}
	\label{table:1d_1hidden}
\end{table}

\begin{table}[h]
		\tiny
		\centering
		\begin{tabular}{|c|c|c|c|c|c|c|c|c|c|c|c|}
			\toprule
			 $\alpha$ &  ResNet&\multicolumn{5}{|c|}{$ReLU^k$ Network} & \multicolumn{5}{|c|}{PSENet}\\
			\cmidrule{3-12}
			 & &k=1 & k=2 & k=3 & k=4 & k=5 &n=1 & n=2 & n=3 & n=4 & n=5\\
			\midrule
			 $2/3$& $3.6\times 10^{-2}$ &$3.0\times 10^{-2}$&$2.6\times 10^{-2}$ &$1.3\times 10^{-2}$ & $3.2\times 10^{18}$ &NaN & $\bm{8.1\times 10^{-3}}$ & $1.1\times 10^{-2}$ & $1.5\times 10^{-2}$ &$9.3\times 10^{-3}$  & $1.3\times 10^{-2}$ \\			\midrule
			 $3/4$& $2.8\times 10^{-3}$ & $1.0\times 10^{-2}$& $7.2\times 10^{-3}$ & $5.8\times 10^{-3}$ & $3.6\times 10^{33}$ & NaN  &$2.9\times 10^{-3}$ &$2.9\times 10^{-3}$ & $\bm{2.7\times 10^{-3}}$ & $4.5\times 10^{-3}$ & $2.8\times 10^{-3}$ \\			\midrule
			 $4/5$& $1.7\times 10^{-3}$& $4.2\times 10^{-3}$ & $2.9\times 10^{-3}$& $2.5\times 10^{-3}$ & $1.8\times 10^{13}$ & NaN & $1.5\times 10^{-3}$& $\bm{1.0\times 10^{-3}}$ & $1.3\times 10^{-3}$ & $4.0\times 10^{-3}$ &$1.6\times 10^{-3}$\\			\bottomrule
		\end{tabular}
\caption{Accuracy comparison of $\int_0^1 (N(x)-f(x))^2+ (N'(x)-f'(x))^2dx $ with $f(x)=x^\alpha$ with $\alpha\in(0,1)$ on $x\in[0,1]$.\red{(NaN stands for Not a number and indicates the training failure.)}}
	\label{table:singular}
\end{table}

\subsection{Comparison with ResNets on different datasets}
We compare the PSENet with different ResNets on CIFAR-10 and CIFAR-100. 
All the deep residual network \red{architectures} considered in our experiment are reported in \cite{he2016deep}. We use the hyperparameters shown in Table \ref{Resnet} \red{to} train the ResNets on CIFAR-10, CIFAR-100, and ImageNet datasets.
Results in Tables \ref{table:cifar_10} and \ref{table:cifar_100} show that the PSENet  achieves better \red{accuracy} rates than ResNet with the same number of layers.
Moreover, the PSENet  achieves better \red{accuracy} rates than ResNet with shallow networks and keeps comparable \red{accuracy} rates with deep networks. 
\red{On the ImageNet dataset , PSENet has a lower error rate on shallow network, such as those with 18 layers, and has a comparable error rate on deep networks, such as those with 34 layers, as shown in Fig. \ref{fig2}.}

\begin{table}[h]
	\centering
\scriptsize
	\begin{tabular}{|c|c|c|}
		\toprule
		Parameter & CIFAR-10 $\&$ CIFAR-100 & ImageNet\\
		\midrule
		Data Augmentation & \{RandomHorizontalFlip $\&$ RandomCrop\} & \{RandomHorizontalFlip $\&$ RandomResizedCrop\}\\
		Number of epochs & 250 & 90\\
		Batch size & 128 & \{ResNet-50: 128, ResNet-34: 256\}\\
		Initial learning rate & 0.2 & 0.1\\
		Learning rate schedule & Decrease by half every 30 epochs & Decrease by $1/10$ every 30 epochs\\
		Bias initialization & Both & False\\
		Number of runs & 10 & 1\\\hline
		Batch normalization & \multicolumn{2}{|c|}{True}\\
		Weight decay & \multicolumn{2}{|c|}{$5\times 10^{-4}$}\\
		Optimizer & \multicolumn{2}{|c|}{SGD with momentum = 0.9} \\
		\bottomrule
	\end{tabular}
\caption{Hyperparameters for the residual networks on CIFAR-10/CIFAR-100 and ImageNet datasets.}
		\label{Resnet}
\end{table}

\begin{table}[h]
		\scriptsize
		\centering
	\caption{Comparison {of percent accuracy} between PSENet and ResNet on the CIFAR-10 dataset with different numbers of layers}
		\begin{tabular}{|c|c|c|c|c|c|c|c|}
			\toprule
			 Number of layers & Coefficient kernel size &  ResNet & \multicolumn{5}{|c|}{PSENet}\\
			\cmidrule{4-8}
			 & & &n=1 & n=2 & n=3 & n=4 & n=5\\
			\midrule
			  4 & 3$\times$3 conv& 76.39\% & 79.44\% & 80.18\% & 80.34\% & {\bf 81.02\%} & 80.37\%\\
			  \midrule
			   6 & 3$\times$3 conv& 83.14\% & 85.40\% & 87.31\% & 87.45\% & {\bf 87.85\%} & 87.78\%\\
			    \midrule
			   8 & 3$\times$3 conv& 87.54\% & 88.79\% & 90.58\% & {\bf 90.76\%} & 90.75\% & 90.66\%\\
			   \midrule
			   14 & 3$\times$3 conv& 91.15\% & 91.86\% & 93.35\% & {\bf 93.50\%} & 93.15\% & 93.31\%\\
			   \midrule
			   20 & 3$\times$3 conv& 92.55\% & 92.82\% & {\bf 93.60\%} & 93.33\% & 92.98\% & 92.22\%\\
			    \midrule
			   26 & 3$\times$3 conv& 93.40\% & 93.40\% & {\bf 93.64\%} & 93.10\% & 93.06\% & 92.39\%\\
			   \midrule
			   56 & 3$\times$3 conv& 94.16\% & 94.16\% & 93.87\% & {\bf 94.16\%} & 93.58\% & 93.15\%\\
			   \midrule
			   110 & 1$\times$1 conv& 94.38\% & 94.38\% & {\bf 94.53\%} & 93.93\% & 93.80\% & 92.02\%\\
			\bottomrule
		\end{tabular}
	\label{table:cifar_10}
\end{table}

\begin{table}[h]
		\scriptsize
		\centering
	\caption{Comparison \red{of percent accuracy} between PSENet and ResNet on the CIFAR-100 dataset with different numbers of layers}
		\begin{tabular}{|c|c|c|c|c|c|c|c|}
			\toprule
			 Numbers of layers & Coefficient kernel size &  ResNet & \multicolumn{5}{|c|}{PSENet}\\
			\cmidrule{4-8}
			 & & &n=1 & n=2 & n=3 & n=4 & n=5\\
			\midrule
			  4 & 3$\times$3 conv & 47.61\% & 48.82\% & 51.10\% & 51.01\% & {\bf 51.53\%} & 51.03\%\\
			  \midrule
			   6 & 3$\times$3 conv & 52.60\% & 55.15\% & 59.33\% & 59.61\% & 59.38\% & {\bf 59.71\%}\\
			    \midrule
			   8 & 3$\times$3 conv & 59.43\% & 62.48\% & 65.90\% & {\bf 66.81\%} & 66.36\% & 66.50\%\\
			   \midrule
			   14 & 3$\times$3 conv & 67.01\% & 67.62\% & {\bf 69.09\%} & 69.05\% & 68.81\% & 68.22\%\\
			   \midrule
			   20 & 3$\times$3 conv & 68.34\% & 68.55\% & {\bf 69.61\%} & 68.93\% & 67.32\% & 65.58\%\\
			    \midrule
			   26  & 3$\times$3 conv & 69.03\% & 67.51\% & {\bf 69.52\%} & 67.38\% & 65.55\% & 61.40\%\\
			   \midrule
			   56  & scalar & 72.70\% & 71.60\% & 71.92\% & 71.91\% & 72.36\% & 72.45\%\\
			   \midrule
			   110  & scalar & 73.56\% & {\bf 74.29\%} & 73.04\% & 74.04\% & 74.26\% & 73.01\%\\
			\bottomrule
		\end{tabular}
	\label{table:cifar_100}
\end{table}

\begin{figure}[h]
	\centering
	\includegraphics[width=0.45\textwidth]{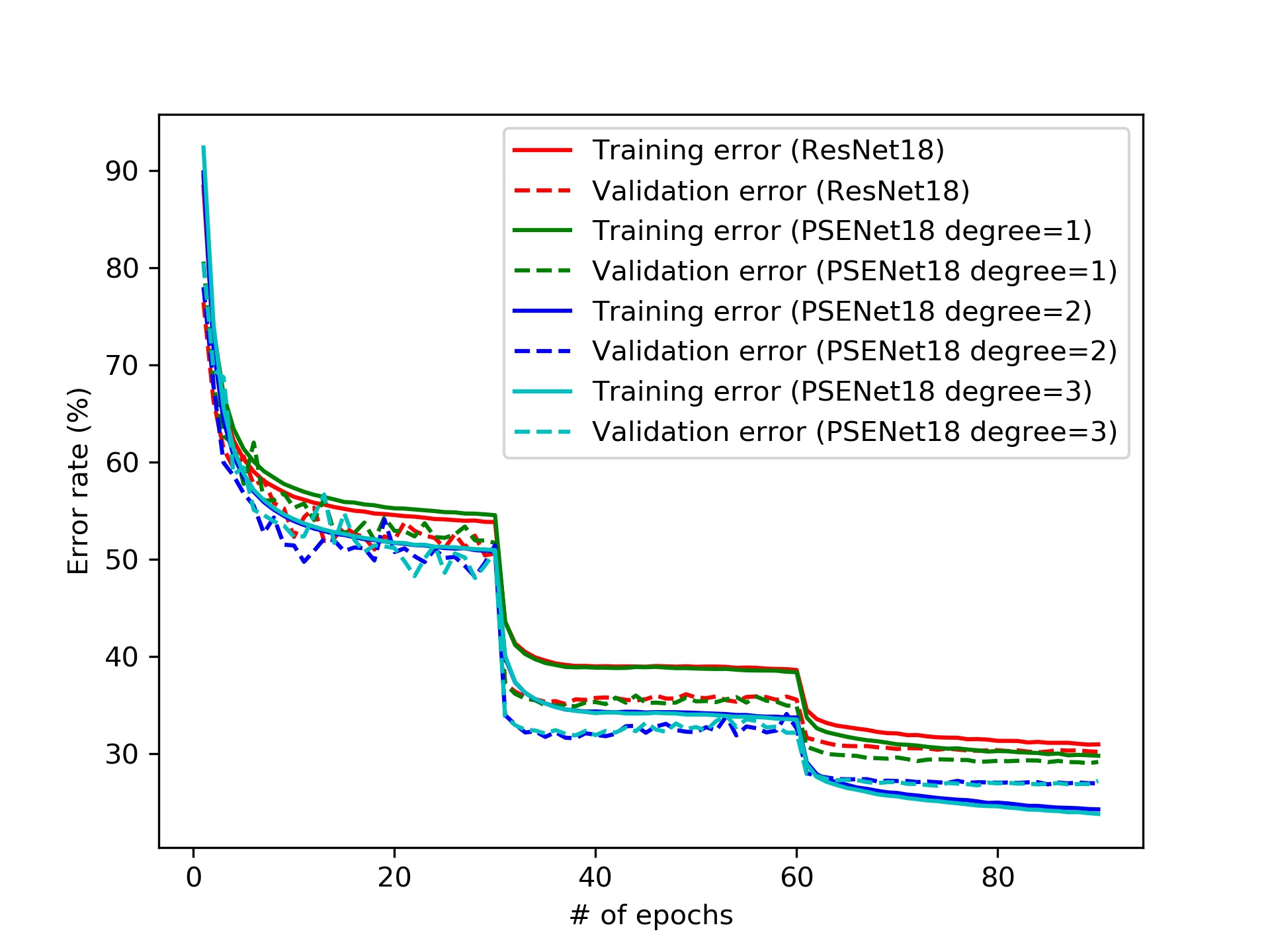}\label{resnet18}
	\includegraphics[width=0.45\textwidth]{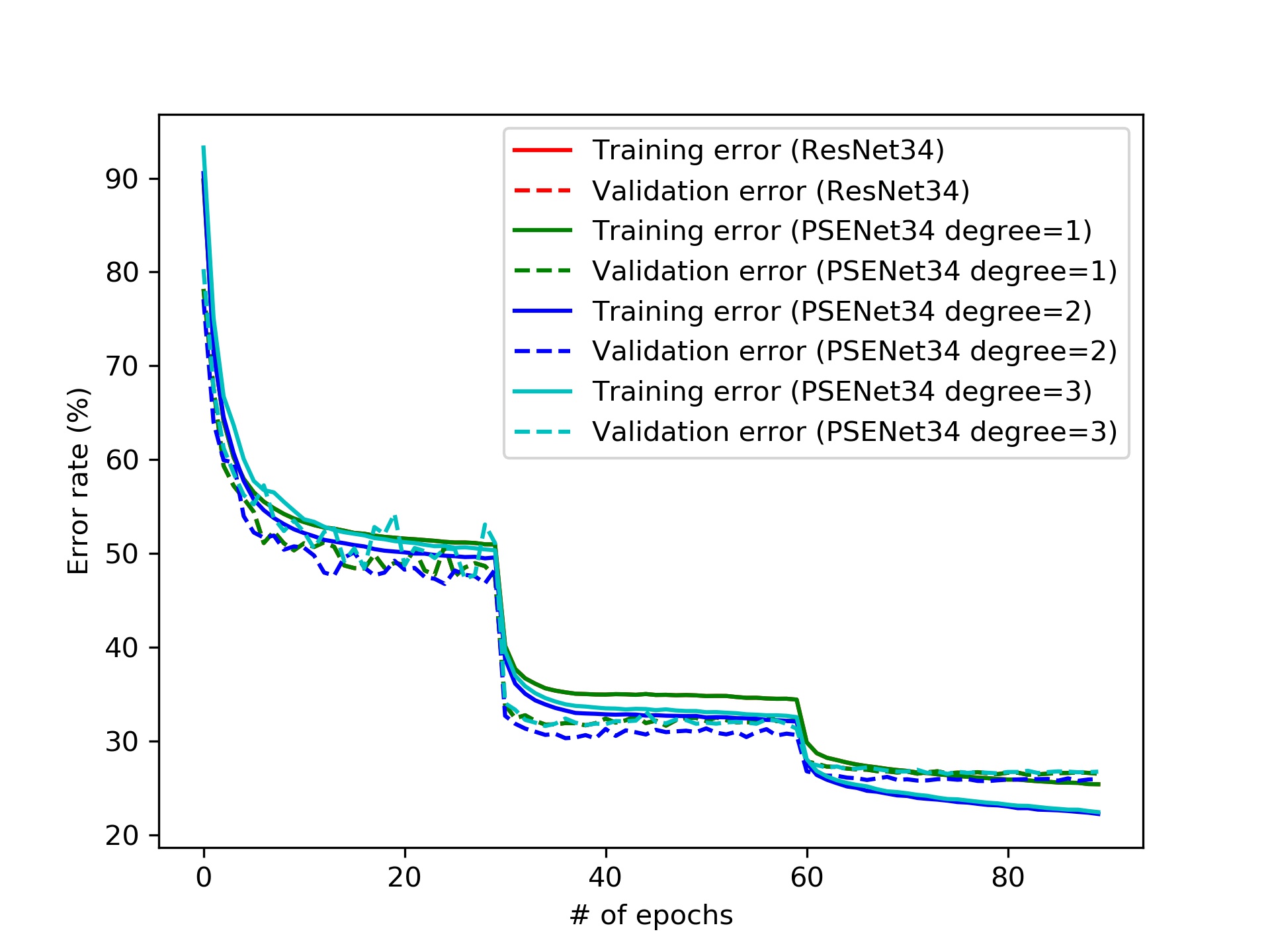}\label{resnet34}
	\caption{Error rates  of image classification on ImageNet for ResNet and PSENet. {\bf Left: } ResNet18 and PSENet18; {\bf right: } ResNet34 and PSENet34. \red{Notice: On the right panel, red lines (ResNet34) are overlapped by green lines (PSENet34 degree=1) and all experimental results shown above come from single optimization runs.}}
	\label{fig2}
\end{figure}

\section{Conclusion}
We develop a novel neural network by combing the ideas of PSE and the neural network approximation. Theoretically, we prove the better approximation result of PSENet by comparing with ReLU$^k$-DNN and the optimal approximation rate on singular functions. Moreover, the PSENet can achieve a better approximation accuracy on the shallow network structure in comparison with other neural networks. Several numerical results have been used to demonstrate the advantages of the PSENet.
This new approach  shows that increasing the degree of PSENet can also lead to further
performance improvements rather than going deep. \red{However, the performance can also decrease when the degree is larger than the optimal degree of PSENet. Obtaining an optimal degree by analysis is one of the future directions. Another interesting avenue to pursue is PSENet with other activation functions rather than ReLU. In this paper, our analysis and numerical results are based on ReLU activation function (or ReLU$^k$). But the novel PSENet architecture can define any activation function. Some challenges, such as weight initialization including $\alpha$ and approximation properties, will also be explored in the future.}

\bibliographystyle{plain}
\bibliography{PSENet}

\begin{thebibliography}{10}

\bibitem{adams2003sobolev}
Robert~A Adams and John~JF Fournier.
\newblock {\em Sobolev spaces}.
\newblock Elsevier, 2003.

\bibitem{arora2018understanding}
R.~Arora, A.~Basu, P.~Mianjy, and A.~Mukherjee.
\newblock Understanding deep neural networks with rectified linear units.
\newblock In {\em International Conference on Learning Representations}, 2018.

\bibitem{arora2018optimization}
S.~Arora, N~Cohen, and E.~Hazan.
\newblock On the optimization of deep networks: Implicit acceleration by
  overparameterization.
\newblock In {\em 35th International Conference on Machine Learning}, 2018.

\bibitem{babuvska1994p}
Ivo Babu{\v{s}}ka and Manil Suri.
\newblock The p and h-p versions of the finite element method, basic principles
  and properties.
\newblock {\em SIAM review}, 36(4):578--632, 1994.

\bibitem{barron1993universal}
A.~Barron.
\newblock Universal approximation bounds for superpositions of a sigmoidal
  function.
\newblock {\em IEEE Transactions on Information theory}, 39(3):930--945, 1993.

\bibitem{chen2019homotopy}
Q.~Chen and W.~Hao.
\newblock A homotopy training algorithm for fully connected neural networks.
\newblock {\em Proceedings of the Royal Society A}, 475(2231):20190662, 2019.

\bibitem{chernov2011exponential}
A.~Chernov, T.~von Petersdorff, and C.~Schwab.
\newblock Exponential convergence of $ hp $ quadrature for integral operators
  with gevrey kernels.
\newblock {\em ESAIM: Mathematical Modelling and Numerical
  Analysis-Mod{\'e}lisation Math{\'e}matique et Analyse Num{\'e}rique},
  45(3):387--422, 2011.

\bibitem{de1971subroutine}
C.~de~Boor.
\newblock Subroutine package for calculating with b-splines.
\newblock {\em Los Alamos Scient. Lab. Report LA-4728-MS}, 1971.

\bibitem{e2018exponential}
W.~E and Q.~Wang.
\newblock Exponential convergence of the deep neural network approximation for
  analytic functions.
\newblock {\em Science China Mathematics}, 61(10):1733--1740, 2018.

\bibitem{guhring2020error}
I.~G{\"u}hring, G.~Kutyniok, and P.~Petersen.
\newblock Error bounds for approximations with deep relu neural networks in
  $w^{s,p}$ norms.
\newblock {\em Analysis and Applications}, 18(05):803--859, 2020.

\bibitem{gui1986theh}
W.~Gui and I.~Babu{\v{s}}ka.
\newblock Theh, p andh-p versions of the finite element method in 1 dimension.
\newblock {\em Numerische Mathematik}, 49(6):613--657, 1986.

\bibitem{he2021approximation}
J.~He, L.~Li, and J.~Xu.
\newblock Approximation properties of relu deep neural networks for smooth and
  non-smooth functions.
\newblock {\em In preparation}, 2021.

\bibitem{he2021dnn}
J.~He, L.~Li, and J.~Xu.
\newblock Dnn with heaviside, relu and requ activation functions.
\newblock {\em In preparation}, 2021.

\bibitem{he2021relu}
J.~He, L.~Li, and J.~Xu.
\newblock Relu deep neural networks from the hierarchical basis perspective.
\newblock {\em arXiv preprint arXiv:2105.04156}, 2021.

\bibitem{he2020relu}
J.~He, L.~Li, J.~Xu, and C.~Zheng.
\newblock Relu deep neural networks and linear finite elements.
\newblock {\em Journal of Computational Mathematics}, 38(3):502--527, 2020.

\bibitem{he2016deep}
K.~He, X.~Zhang, S.~Ren, and J.~Sun.
\newblock Deep residual learning for image recognition.
\newblock In {\em Proceedings of the IEEE conference on computer vision and
  pattern recognition}, pages 770--778, 2016.

\bibitem{jones1992simple}
L.~Jones et~al.
\newblock A simple lemma on greedy approximation in hilbert space and
  convergence rates for projection pursuit regression and neural network
  training.
\newblock {\em The annals of Statistics}, 20(1):608--613, 1992.

\bibitem{kang2021ident}
.~Kang, S, W.~Liao, and Y.~Liu.
\newblock Ident: Identifying differential equations with numerical time
  evolution.
\newblock {\em Journal of Scientific Computing}, 87(1):1--27, 2021.

\bibitem{lei2020machine}
H.~Lei, L.~Wu, and E~W.
\newblock Machine-learning-based non-newtonian fluid model with molecular
  fidelity.
\newblock {\em Physical Review E}, 102(4):043309, 2020.

\bibitem{lu2020deep}
J.~Lu, Z.~Shen, H.~Yang, and S.~Zhang.
\newblock Deep network approximation for smooth functions.
\newblock {\em arXiv preprint arXiv:2001.03040}, 2020.

\bibitem{lu2021deepxde}
L.~Lu, X.~Meng, Z.~Mao, and G.~Karniadakis.
\newblock Deepxde: A deep learning library for solving differential equations.
\newblock {\em SIAM Review}, 63(1):208--228, 2021.

\bibitem{lu2017expressive}
Z.~Lu, H.~Pu, F.~Wang, Z.~Hu, and L.~Wang.
\newblock The expressive power of neural networks: A view from the width.
\newblock In {\em Advances in Neural Information Processing Systems}, pages
  6231--6239, 2017.

\bibitem{montanelli2019deep}
H.~Montanelli, H.~Yang, and Q.~Du.
\newblock Deep relu networks overcome the curse of dimensionality for
  bandlimited functions.
\newblock {\em arXiv preprint arXiv:1903.00735}, 2019.

\bibitem{montufar2014number}
G.~Montufar, R.~Pascanu, K.~Cho, and Y.~Bengio.
\newblock On the number of linear regions of deep neural networks.
\newblock In {\em Advances in neural information processing systems}, pages
  2924--2932, 2014.

\bibitem{opschoor2020deep}
J.~Opschoor, P.~Petersen, and C.~Schwab.
\newblock Deep relu networks and high-order finite element methods.
\newblock {\em Analysis and Applications}, pages 1--56, 2020.

\bibitem{opschoor2019exponential}
J.~Opschoor, C.~Schwab, and J.~Zech.
\newblock Exponential relu dnn expression of holomorphic maps in high
  dimension.
\newblock {\em SAM Research Report}, 2019, 2019.

\bibitem{poggio2017and}
T.~Poggio, H.~Mhaskar, L.~Rosasco, B.~Miranda, and Q.~Liao.
\newblock Why and when can deep-but not shallow-networks avoid the curse of
  dimensionality: a review.
\newblock {\em International Journal of Automation and Computing},
  14(5):503--519, 2017.

\bibitem{schwab1998numerical}
C.~Schwab.
\newblock {\em p- and hp- Finite Element Methods: Theory and Applications to
  Solid and Fluid Mechanics}.
\newblock The Clarendon Press, Oxford University Press New York, 1998.

\bibitem{shen2011spectral}
Jie Shen, Tao Tang, and Li-Lian Wang.
\newblock {\em Spectral methods: algorithms, analysis and applications},
  volume~41.
\newblock Springer Science \& Business Media, 2011.

\bibitem{shen2019deep}
Z.~Shen, H.~Yang, and S.~Zhang.
\newblock Deep network approximation characterized by number of neurons.
\newblock {\em arXiv preprint arXiv:1906.05497}, 2019.

\bibitem{shen2019nonlinear}
Z.~Shen, H.~Yang, and S.~Zhang.
\newblock Nonlinear approximation via compositions.
\newblock {\em Neural Networks}, 119:74--84, 2019.

\bibitem{siegel2020high}
J.~Siegel and J.~Xu.
\newblock High-order approximation rates for neural networks with reluk
  activation functions.
\newblock {\em arXiv preprint arXiv:2012.07205}, 2020.

\bibitem{tang2019chebnet}
S.~Tang, B.~Li, and H.~Yu.
\newblock Chebnet: Efficient and stable constructions of deep neural networks
  with rectified power units using chebyshev approximations.
\newblock {\em arXiv preprint arXiv:1911.05467}, 2019.

\bibitem{telgarsky2016benefits}
M.~Telgarsky.
\newblock Benefits of depth in neural networks.
\newblock {\em Journal of Machine Learning Research}, 49(June):1517--1539,
  2016.

\bibitem{wang2021laplacian}
B.~Wang, D.~Zou, Q.~Gu, and S.~Osher.
\newblock Laplacian smoothing stochastic gradient markov chain monte carlo.
\newblock {\em SIAM Journal on Scientific Computing}, 43(1):A26--A53, 2021.

\bibitem{xu2020finite}
J.~Xu.
\newblock The finite neuron method and convergence analysis.
\newblock {\em arXiv preprint arXiv:2010.01458}, 2020.

\bibitem{yaacov2014multivariate}
I.~Yaacov.
\newblock A multivariate version of the vandermonde determinant identity.
\newblock {\em arXiv preprint arXiv:1405.0993}, 8, 2014.

\bibitem{zhu2019stop}
W.~Zhu, Q.~Qiu, B.~Wang, J.~Lu, G.~Sapiro, and I.~Daubechies.
\newblock Stop memorizing: A data-dependent regularization framework for
  intrinsic pattern learning.
\newblock {\em SIAM Journal on Mathematics of Data Science}, 1(3):476--496,
  2019.

\end{thebibliography}

\end{document}